\documentclass[11pt]{article}
\usepackage[margin=1in]{geometry} 
\usepackage[T1,T2A]{fontenc}
\usepackage[utf8]{inputenc}
\usepackage{amssymb,amsfonts,amsmath,bbm,mathrsfs,stmaryrd}
\usepackage{xcolor}
\usepackage{url}
\usepackage{relsize}
\usepackage{enumerate}
\usepackage{tikz-cd}
\usepackage{tikz}
\usepackage{subcaption}
\usepackage{mathtools}
\usepackage{quiver}
\usepackage{extarrows}

\usepackage[cal=boondoxo]{mathalfa}
\usepackage[backend=biber, style=alphabetic, url=false]{biblatex}
\usepackage[colorlinks,
             linkcolor=black!75!red,
             citecolor=blue,
             pdftitle={},
             pdfauthor={x},
             pdfproducer={pdfLaTeX},
             pdfpagemode=FullScreen,
             bookmarksopen=true,
             bookmarksnumbered=true]{hyperref}
\usepackage{tikz}
\usetikzlibrary{arrows,calc,decorations.pathreplacing,decorations.markings,intersections,shapes.geometric,through,fit,shapes.symbols,positioning,decorations.pathmorphing}

\usepackage[amsmath,thmmarks]{ntheorem}
\usepackage{cleveref}
\usepackage{comment}

\creflabelformat{enumi}{label*=\arabic*}

\crefname{section}{Section}{Sections}
\crefformat{section}{#2Section~#1#3} 
\Crefformat{section}{#2Section~#1#3} 

\crefname{subsection}{\S}{\S\S}
\crefformat{subsection}{#2\S#1#3} 
\Crefformat{subsection}{#2\S#1#3} 

\theoremstyle{plain}

\newtheorem{lemma}{Lemma}[section]
\newtheorem{proposition}[lemma]{Proposition}

\newtheorem{theorem}[lemma]{Theorem}

\theoremstyle{nonumberplain}

\theoremstyle{plain}
\theorembodyfont{\upshape}
\theoremsymbol{\ensuremath{\blacklozenge}}

\newtheorem{definition}[lemma]{Definition}
\newtheorem{example}[lemma]{Example}
\newtheorem{remark}[lemma]{Remark}

\crefname{definition}{definition}{definitions}
\crefformat{definition}{#2definition~#1#3} 
\Crefformat{definition}{#2Definition~#1#3} 

\crefname{ex}{example}{examples}
\crefformat{example}{#2example~#1#3} 
\Crefformat{example}{#2Example~#1#3} 

\crefname{remark}{remark}{remarks}
\crefformat{remark}{#2remark~#1#3} 
\Crefformat{remark}{#2Remark~#1#3} 

\crefname{convention}{convention}{conventions}
\crefformat{convention}{#2convention~#1#3} 
\Crefformat{convention}{#2Convention~#1#3} 

\crefname{claim}{claim}{claims}
\crefformat{claim}{#2claim~#1#3} 
\Crefformat{claim}{#2Claim~#1#3} 

\crefname{conjecture}{conjecture}{conjectures}
\crefformat{conjecture}{#2conjecture~#1#3} 
\Crefformat{conjecture}{#2Conjecture~#1#3}

\crefname{lemma}{lemma}{lemmas}
\crefformat{lemma}{#2lemma~#1#3} 
\Crefformat{lemma}{#2Lemma~#1#3} 

\crefname{proposition}{proposition}{propositions}
\crefformat{proposition}{#2proposition~#1#3} 
\Crefformat{proposition}{#2Proposition~#1#3} 

\crefname{question}{question}{questions}
\crefformat{question}{#2question~#1#3} 
\Crefformat{question}{#2Question~#1#3}

\crefname{corollary}{corollary}{corollaries}
\crefformat{corollary}{#2corollary~#1#3} 
\Crefformat{corollary}{#2Corollary~#1#3} 

\crefname{theorem}{theorem}{theorems}
\crefformat{theorem}{#2theorem~#1#3} 
\Crefformat{theorem}{#2Theorem~#1#3} 

\crefname{assumption}{assumption}{Assumptions}
\crefformat{assumption}{#2assumption~#1#3} 
\Crefformat{assumption}{#2Assumption~#1#3} 

\crefname{equation}{}{}
\crefformat{equation}{(#2#1#3)} 
\Crefformat{equation}{(#2#1#3)}

\theoremstyle{nonumberplain}
\theoremsymbol{\ensuremath{\blacksquare}}

\newtheorem{proof}{Proof}

\newcommand{\vcong}{\mathrel{\text{\rotatebox[origin=c]{90}{$\cong$}}}}

\DeclareUnicodeCharacter{2217}{*}
\DeclareUnicodeCharacter{1D436}{$C$}

\def\polhk#1{\setbox0=\hbox{#1}{\ooalign{\hidewidth
    \lower1.5ex\hbox{`}\hidewidth\crcr\unhbox0}}}

\newcommand{\bes}{\begin{equation*}}
\newcommand{\ees}{\end{equation*}}
\newcommand{\be}{\begin{equation}}
\newcommand{\ee}{\end{equation}}

\allowdisplaybreaks

\usetikzlibrary{cd}
\usetikzlibrary{arrows}
\newsavebox{\boxA}
\newsavebox{\boxB}
\newsavebox{\boxC}
\newsavebox{\boxD}
\newsavebox{\boxE}
\newsavebox{\boxF}
\begin{document}

\baselineskip=15pt

\title{Quantum real projective spaces as quantum CW-complexes via topological graphs}
\author{
  Atul Gothe
}

\date{}

\newcommand{\Addresses}{{
  \footnotesize  
  
  \textsc{Instytut Matematyczny Polskiej Akademii Nauk, ul. \'Sniadeckich 8, 00-656 Warszawa, Poland}\par\nopagebreak
  \textsc{Instytut Matematyki, Uniwersytet Warszawski, ul. Stefana Banacha 2, 02-097 Warszawa, Poland}\par\nopagebreak
  \textit{E-mail address}:
  \texttt{a.gothe@uw.edu.pl}
}}

\maketitle

\begin{abstract}
We use the notion and properties of topological graphs to construct a quantum CW complex for quantum real projective spaces. This construction mirrors the CW complex construction in the classical case. Moreover, we introduce multipullback even-spheres as boundaries of noncommutative "cylinders" and show that they admit a gluing akin to the gluing of balls to the ends of a hollow closed cylinder.
\end{abstract}


\vspace{.5cm}

\tableofcontents

\section{Introduction}
Classically, a CW-complex is a topological space which is obtained by gluing together balls of different dimensions in a specific way. They provide a convenient category of topological spaces to work with, and many of the classical spaces admit a CW structure. There have been numerous attempts to generalise CW-complexes to the realm of noncommutative topology, most notably by Eilers-Loring-Pedersen in \cite{pedersen_stability_1998} and more recently by \cite{dandrea_quantum_2026}. A lot of $C^*$-algebras admit a quantum CW complex structure, most notably, Podleś sphere, quantum complex projective spaces, quantum teardrops, etc, we refer to \cite{dandrea_quantum_2026} for an extensive survey. $C^*$-algebras of quantum real projective spaces were first defined by Hong-Szymański \cite{hong_quantum_2002}, and furthermore they show that they admit a directed graph structure. Similarly, by the means of a quantum double suspension, Hong-Szymański, defined the $C^*$-algebras of quantum balls \cite{hong_noncommutative_2008}. The even-dimensional balls admit a presentation in terms of $C^*$-algebras of directed graphs, however,the odd-dimensional balls cannot admit such a presentation \cite{hong_noncommutative_2008}\cite{foster_quantum_2005}. However, in \cite{gothe_gluing_2026}, it was shown that they do admit a presentation in terms of topological graphs. 

Topological graphs and their $C^*$-algebras were introduced by Katsura in \cite{katsura_class_2004} and extend the concept of directed graphs to locally compact topologies on the edge and the vertex spaces. One can construct a $C^*$-correspondence corresponding to the data of a topological graph. The $C^*$-algebras corresponding to topological graphs are precisely the Cuntz-Pimsner algebras originating from these $C^*$-correspondences. As examples they include graph $C^*$-algebras and homeomorphism $C^*$-algebras. They have fully classified gauge-invariant ideal structure and their K-groups can be computed from a six term exact sequence.

We utilise the concept and properties of topological graph $C^*$-algebras to give a quantum CW-complex structure of the $C^*$-algebras of quantum real projective spaces. The maps constructed mirror the construction of CW-complexes for classical real projective spaces and provide credence to the ideas presented.

The mulitipullback $C^*$-algebras were defined in \cite{pedersen_pullback_1999}, as a way to dualise gluings with multiple legs. As a result they provide new examples of $C^*$-algebras such as the multipullback quantum odd-dimensional spheres and multipullback quantum complex projective spaces \cite{hajac_k-theory_2018}. The latter, in particular also admit a quantum CW decomposition. Building on this, we define the $C^*$-algebra of multipullback quantum even-spheres and show that they admit a mutipullback structure akin to the classical construction of a suspension by gluing cones over the base space.

\section{Topological graphs}
We start with the definition of a topological graph. Topological graphs were introduced by Katsura in \cite{katsura_class_2004} as a way to generalise $C^*$-algebras originating from directed graphs and topological dynamical system. As it turns out, they have a rich structure and many properties of the $C^*$-algebras can be written in terms of the topological and combinatorial properties of the underlying graphs. And although they were originally defined to generalise graph $C^*$-algebras, many well-known $C^*$-algebras which cannot admit a discrete graph presentation can admit a topological graph presentation.

 A $\mathbf{topological \ graph}$ is a quadruple $(E^0,E^1,s,r)$ such that $E^0,E^1$ are locally compact spaces $s ,r : E^1 \xrightarrow{} E^0$ are continuous maps such that $r$ is further a local homeomorphism.

Before moving further with topological graphs, we will provide a primer for the construction of Cuntz-Pimsner algebras from $C^*$-correspondences. 

\begin{definition}
A Hilbert $A$-module X is a $C^*$-correspondence if it is equipped with a $*$-homomorphism $\varphi_X : A \xrightarrow[]{} \mathcal{L}(X)$.
\end{definition}

\begin{definition}
A representation $(\psi^1, \psi^0)$ of a $C^*$-correspondence on a $C^*$-algebra $B$ is a pair of maps $\psi^0 : A \xrightarrow[]{} B$ is a $*$-homomorphisms and $\psi^1: X \xrightarrow[]{} B$ is a linear map such that the following natural conditions hold:
\begin{enumerate}
    \item $\psi^1(x)^* \psi^1(y) = \psi^0(\langle x , y\rangle_A)$
    \item $\psi^0(a)\psi^1(x) = \psi^1(\varphi_X(a) x)$
\end{enumerate}
\end{definition}

To define the Cuntz-Pimsner algebra, we need the notion of a Katsura ideal -
\begin{definition}
Let $X$ be a $C^*$-correspondence over $A$. Then the Katsura ideal $J_X$ is:

\begin{align*}
    J_X &= \varphi_X^{-1} (\mathcal{K(X)}) \cap (ker\varphi_X)^{\perp}  \\
        &= \{ a \in A \ | \ \varphi_X(a) \in \mathcal{K}(X) \text{ and } ab = 0 \text{ for all } b \in ker\varphi_X \}
\end{align*}
\end{definition}

\begin{definition}(\cite[Definition~3.4]{katsura_c-algebras_2004})
A representation $(\psi^1,\psi^0)$ of a C*-correspondence $X$ over $A$ on a C*-algebra $B$ is {\em covariant} if $\psi^0(a)=\psi^{(1)}(\varphi_X(a))$ for all $a\in J_X$, where $\psi^{(1)}(\theta_{\eta,\xi}) = t^1(\xi)^*t^1(\eta)$ 
\end{definition}

\begin{definition}(\cite[Definition~3.5]{katsura_c-algebras_2004}, cf.~\cite{voiculescu_free_1996})
The {\em Cuntz--Pimsner algebra} $\mathcal{O}_X$ associated to a correspondence $X$ over $A$ is the C*-algebra $C^*(t^1_X,t^0_X)$ generated by the universal covariant representation $(t^1_X,t^0_X)$ of $(X,A)$. 
\end{definition}
By universality, for every covariant representation $(\psi^1,\psi^0)$ of $X$ on $B$, there is a~unique $*$-ho\-mo\-mor\-phism $\psi:\mathcal{O}_X\to B$ onto the C*-algebra $C^*(\psi^1,\psi^0)$ (generated by the images of $\psi^1$ and $\psi^0$) such that $\psi\circ t^1_X=\psi^1$ and $\psi\circ t^0_X=\psi^0$. Next, the formulas
\begin{equation}\label{cuntzgauge}
\alpha_z(t^0_X(a)):=t^0_X(a),\qquad \alpha_z(t^1_X(\xi)):=zt^1_X(\xi),\qquad z\in\mathbb{T},\quad a\in A,\quad \xi\in X,
\end{equation}
give rise to an action $\alpha:\mathbb{T}\to{\rm Aut}(\mathcal{O}_X)$, called the~{\em gauge action}. A  $*$-homomorphism between Cuntz--Pimsner algebras which is $\mathbb{T}$-equivariant with respect to gauge actions is called {\em gauge-equivariant}. If $I$ is an ideal of $\mathcal{O}_X$ such that $\alpha_z(I)\subseteq I$ for all $z\in\mathbb{T}$, then $I$ is {\em gauge-invariant}.

\subsection{Topological graphs}

We refer the reader to~\cite{katsura_class_2004,katsura_class_2006,katsura_class_2006-1} for an extensive study of topological graphs and their C*-algebras (note, however, that we use a different source-range convention). A \textit{topological graph} $E$ is a~quadruple $(E^0,E^1,s_E,r_E)$, where $E^0$ and $E^1$ are locally compact spaces and $s_E,r_E:E^1\to E^0$ are continuous maps such that $r_E$ is furthermore a local homeomorphism. $E^0$ is referred to as the space of vertices while $E^1$ is the edge space. $s_E$ and $r_E$ are called source and range maps respectively.

Now we will give a brief description of the construction of $C^*$-correspondence associated to a topological graph. To fix notation, denote $A_E \vcentcolon C_0(E^0)$. We define the following $A$-valued bilinear function on $C(E^1)$:
\begin{equation}\label{inner}
\langle \xi,\eta\rangle(v) \vcentcolon \sum_{e\in r_E^{-1}(v)}\overline{\xi(e)}\eta(e).
\end{equation}

Next, define the following subspace of the space of compactly supported functions on $E^1$:
\[
X_E:=\left\{\xi\in C_c(E^1)~:~\langle\xi,\xi\rangle\in A_E\right\}.
\] 
With this definition in hand $(X_E,A)$ is a C*-correspondence with the $A$-valued Hilbert inner product~\eqref{inner} and the left and the right multiplication given by composition with the source map and the range map respectively, i.e.:
\[
(\xi\cdot f)(e):=\xi(e)f(r_E(e)),\qquad (f\cdot\xi)(e):=f(s_E(e))\xi(e),\qquad f\in A,\quad \xi\in X_E.
\]
The $*$-homomorphism $\varphi_E : A_E \xrightarrow[]{} \mathcal{L}(X_E)$ is given by the left action.

The \textit{topological graph algebra} $C^*(E)$ is defined as the Cuntz-Pimsner of the $C^*$-correspondence $(X_E,A_E)$.

\begin{example}
If $E^0, E^1$ are discrete, we recover the usual graph $C^*$-algebras.
\end{example}
\begin{example}
If $X$ is a locally compact space and $\sigma : X \xrightarrow[]{} X$ is a homeomorphism on $X$, then the topological graph algebra corresponding to the topological graph $(X,X, s= id_X, r = \sigma)$ is the homeomorphism $C^*$-algebra $C_0(X) \rtimes \mathbb{Z}$. 
\end{example}

We have the following subsets of $E^0$:
\begin{align*}
E_{\rm fin}^0&:=\{v\in E^0~:~\text{there is a neighborhood $V$ of $v$ such that $s_E^{-1}(V)$ is compact}\},\\
E^0_{\rm inf}&:=E^0\setminus E^0_{\rm fin},\qquad E_{\rm sink}^0:=E^0\setminus \overline{s_E(E^1)},\\
E_{\rm rg}^0&:=E^0_{\rm fin}\setminus \overline{E^0_{\rm sink}},\qquad E_{\rm sg}^0:=E^0\setminus E^0_{\rm rg}.
\end{align*}

A topological graph $E$ is called \textit{row-finite} if $s_E(E^1)=E^0_{\rm rg}$~\cite[Definition~3.22]{katsura_class_2006-1}. This generalises the notion of row-finiteness of discrete graphs~(see, e.g.~\cite{bates_c-algebras_2000}) but is not equivalent to the naive condition $F^0=F^0_{\rm fin}$ for general topological graphs (see Examples~3.21 and~3.23 in~\cite{katsura_class_2006-1}).

To characterise the ideal structure of topological graph algebras in terms of certain combinatorial and topological properties of $E^0$, we define the following subsets:

A subset $Y^0 \subseteq E^0$ is called \textit{positively invariant} if $r_E(e)\in Y^0$ implies $s_E(e)\in Y^0$ for every $e\in E^1$ and \textit{negatively invariant} if for all $v\in Y^0\cap E^0_{\rm rg}$ there exists an $e\in E^1$ such that $s_E(e)=v$ and $r_E(e)\in Y^0$. Furthermore, $Y^0$ is \textit{invariant} if it is both positively and negatively invariant.
\begin{definition}
(\cite[Definition~2.3]{katsura_class_2006-1})
An \textit{admissible pair} $\rho = (Y^0,Z^0)$ is a pair of closed subsets of $E^0$ such that $Y^0$ is invariant and $Z^0$ is such that $Y^0_{\rm sg}\subseteq Z^0\subseteq E^0_{\rm sg}\cap Y^0$.
\end{definition}

Let $I$ be a gauge-invariant ideal of $C^*(E)$. By~\cite[Proposition~2.8]{katsura_class_2006-1}, one can associate an admissible pair to $I$ in the following way:
\begin{align}\label{admideal}
Y^0_I&:=\{v\in E^0~:~f(v)=0\text{ for all }f\in C_0(E^0)\text{ such that }t^0_E(f)\in I\},\\
Z^0_I&:=\{v\in E^0~:~f(v)=0\text{ for all }f\in C_0(E^0)\text{ such that }t^0_E(f)\in I+(t^1_E)^{(1)}(\mathcal{K}(X_E))\}.\nonumber
\end{align}
then $\rho = (Y^0_I, Z^0_I)$ is an admissible pair.

Conversely, given an admissible pair $\rho=(Y^0,Z^0)$, there exists a gauge-invariant ideal $I_\rho$ (see~\cite[Proposition~3.5]{katsura_class_2006-1}).
 
\begin{theorem}\label{idealclass}
{\rm (\cite[Theorem~3.19]{katsura_class_2006-1})}
For a topological graph $E$, the assignment $I\mapsto \rho_I=(Y^0_I,Z^0_I)$ and $\rho=(Y^0,Z^0)\mapsto I_\rho$ give rise to an inclusion-reversing one-to-one correspondence between the set of gauge-invariant ideals of $C^*(E)$ and the set of admissible pairs in $E^0$.
\end{theorem}

\begin{definition}(\cite[Definition~5.1]{katsura_class_2006-1})
Let $V \subseteq E^0$. Then $V$ is called \textit{hereditary} if $r(s^{-1}(V)) \subseteq V$ and \textit{saturated} if for $v \in E^0_{rg}$ with $r(s^{-1}(v) \subseteq V$ implies $v \in V$. For $V \subseteq E^0$, define:
\begin{align*}
    H(V) = \bigcup_{n=0}^{\infty} r^n((s^{n})^{-1}(V))
\end{align*}
and,
$S(V) = \bigcup V_k$ with $V_0 = V$ and $V_{k+1}$ defined inductively as:
\begin{align*}
    V_{k+1} = V_k \cup \{ v \in E^0_{rg} \vert r(s^{-1}(v)) \subseteq V \}
\end{align*}
\label{hss}
\end{definition}

Next we need to have a notion of morphisms between topological graphs. 
\begin{definition}\label{factorkatsura}
(\cite[Definition~2.1, Definition~2.6]{katsura_class_2006})
Let $E$ and $F$ be topological graphs. Then, a \textit{regular factor map} $m : E \xrightarrow[]{} F$ is a pair of proper continuous maps, $m^1:E^1\xrightarrow[]{} F^1$ and $m^0:E^0 \xrightarrow[]{} F^0$ such that the following hold:
\begin{enumerate}
\item[(F1)] $r_F(m^1(e))=m^0(r_E(e))$ and $s_F(m^1(e))=m^0(s_E(e))$ for all $e\in E^1$.
\item[(F2)] If $x\in F^1$ and $v\in E^0$ satisfies $r_F(x)=m^0(v)$, then there exists a unique element $e\in E^1$ such that $m^1(e)=x$ and $r_E(e)=v$.
\item[(F3)] $m^0(E^0_{\rm sg})\subseteq F^0_{\rm sg}$.
\end{enumerate}
\end{definition}
\begin{remark}
The definition of regularity in~\cite[Definition~2.6]{katsura_class_2006} is slightly different. However, it is equivalent to our definition due to~\cite[Lemma~2.7]{katsura_class_2006}.
\end{remark}
Given a regular factor map $m:E\to F$, we get a pair of contravariantly induced maps:
\begin{equation}\label{precomp}
\mu^0:A_F\to A_E,\qquad (\mu^0(f))(v):=f(m^0(v)),\qquad
\mu^1:X_F\to X_E,\qquad (\mu^1(\xi))(e):=\xi(m^1(e)).
\end{equation}
The following result justifies the definition of a regular factor map.
\begin{proposition}\label{functc*}
{\rm (\cite[Proposition~2.9]{katsura_class_2006})}
Let $m:E\to F$ be a regular factor map. Then there exists a unique gauge-equivariant $*$-homomorphism $\mu:C^*(F)\to C^*(E)$ such that $\mu\circ t_F^i=\mu^i$, $i=0,1$, where $\mu^0$ and $\mu^1$ are defined by~\eqref{precomp}.
\end{proposition}

An important result for directed graph algebras is that of a Cuntz-Krieger uniqueness theorem. This allows one to construct a faithful image of a graph algebra inside a $C^*$-algebra. The counterpart of this theorem for the case of topological graphs is extremely useful as well and needs a few more preliminaries.

A path in a topological graph is a tuple of edges such that the range of an edge is the source of the proceeding edge. More precisely, a \textit{path} $e = (e_1,e_2,\dots,e_n) \in (E^1)^{\times n}$ is a tuple such that $r(e_i) = s(e_{i+1})$ for all $1\leq i \leq n-1$. The set of all paths of length $n$ is denoted by $E^n$.
\begin{definition}
A path in $e = (e^1,e^2, \dots, e^n)$ in $E^n$ with $n \geq 1$ is a \textit{loop} if $r(e^n) = s(e^1)$. A vertex in $e$ is called a base point. A loop is said to be \textit{without exits} if $s^{-1}(s(e^k)) = e^k$.
\end{definition}
\begin{definition}
A topological graph $E$ is called \textit{topologically free} if the set of base points of loops without exits has empty interior.
\end{definition}
\begin{remark}
Similar results for general Cuntz-Pimsner algebras appear in \cite{paulovicks_cuntz-krieger_2025}, referred to as Condition (S).  
\end{remark}
For topologically free topological graphs we have the following Cuntz-Krieger uniqueness theorem:
\begin{theorem}(\cite[Theorem 5.12]{katsura_class_2004}
If $E$ is topological graph such that it is topologically free and $B$ is any $C^*$-algebra. Then if $\rho : C^*(E) \xrightarrow[]{} B$ is a $*$-homomorphism such that $\rho \vert_{t^0(C_0(E^0))}$ is injective, then $\rho$ is injective.
\end{theorem}

\section{Quantum CW-complex}
The theory of CW-complexes for noncommutative spaces was first explored by Eilers-Loring-Pedersen in \cite{pedersen_stability_1998} and more recently in a generalised framework by D'Andrea-Hajac-Maszczyk-Zielinski in \cite{dandrea_quantum_2026}. We refer to the latter as the general reference for this section on quantum CW-complexes and a plethora of examples therein.
\begin{definition}(\cite[Definition 2.1]{dandrea_quantum_2026})
An \textit{unpointed Waldhausen category} is a category $\mathcal{C}$ with an initial object $\empty$ and a terminal object $\star$, with two distinguished classes of morphisms: $Cof$ of \textit{cofibrations} denoted by $\xhookrightarrow{}$ and $Weq$ of \textit{weak equivalences} denoted by $\xrightarrow[]{\sim}$ such that the following hold:
\begin{enumerate}
\item[(Iso)] Isomorphisms are both cofibrations and weak equivalences
\item[(Com)] $Cof$ and $Weq$ are closed under composition
\item[(Cof)] For any object, the unique morphism $\empty \xhookrightarrow{} X$ is a cofibration.  Furthermore, if $X \xhookrightarrow{} Y$ is a cofibration and $X \xrightarrow[]{} \tilde{X}$ is any morphism then the pushout $\tilde{X} \xhookrightarrow{} \tilde{X} \bigcup_X Y$ is also a cofibration.
\item[(Weq)] Given any commutative diagram of the form:
\[\begin{tikzcd}[ampersand replacement=\&]
	Z \& X \& Y \\
	{\tilde{Z}} \& {\tilde{X}} \& {\tilde{Y}}
	\arrow["{\sim}", from=1-1, to=2-1]
	\arrow[from=1-2, to=1-1]
	\arrow[hook, from=1-2, to=1-3]
	\arrow["{\sim}", from=1-2, to=2-2]
	\arrow["{\sim}", from=1-3, to=2-3]
	\arrow[from=2-2, to=2-1]
	\arrow[hook, from=2-2, to=2-3]
\end{tikzcd}\]
the induced maps on the pushouts $Z \bigsqcup_X Y \xrightarrow[]{} \tilde{Z} \bigsqcup_{\tilde{X}} \tilde{Y}$ is a weak equivalence.
\end{enumerate}
\end{definition}
\begin{definition}(\cite[Definition 2.2]{dandrea_quantum_2026})
An unpointed Waldhausen category is a \textit{cw-Waldhausen category} such that the following are satisfied:
\begin{enumerate}
\item[(cw-W)] 
For every pushout diagram,
    \[\begin{tikzcd}[ampersand replacement=\&]
	\& {\tilde{\tilde{Z}}} \& \\
	{\tilde{Y}} \&\& {\tilde{Z}} \\
	\& Y
	\arrow["{\tilde{j}}", from=2-1, to=1-2]
	\arrow["{\tilde{h}}"', from=2-3, to=1-2]
	\arrow["g", from=3-2, to=2-1]
	\arrow["j"', hook, from=3-2, to=2-3]
\end{tikzcd}\]
 with $j$ being a cofibration, then $\tilde{h}$ is a weak equivalence iff g is one as well.
 \item[(2/3)] For every commutative diagram of the form
 \[
 \begin{tikzcd}[ampersand replacement=\&]
	X \& Y \& Z
	\arrow[from=1-1, to=1-2]
	\arrow[curve={height=12pt}, from=1-1, to=1-3]
	\arrow[from=1-2, to=1-3]
\end{tikzcd}
\]
with any two of the morphisms being weak equivalences implies the third one sn a weak equivalence as well.
\end{enumerate}
\end{definition}
This allows one to work in the homotopy category- $Ho(\mathcal{C})$.
\begin{definition}(\cite[Definition 2.3]{dandrea_quantum_2026})
An object $X$ in a cw-Waldhausen category is weakly contractible if the unique morphism $X \xrightarrow[]{} \star$ is a weak equivalence.
\end{definition}
\begin{definition}
A \textit{cw-Waldhausen symmetric monoidal category} $\mathcal{C}$ is a category which is both a cw-Waldhausen category and a symmetric monoidal category with the terminal object as a monoidal unit together with a fully distinguished subcategory $\mathcal{C}^{\times}$ such that the following are true:
\begin{enumerate}
\item[(Con-Weq)] Any morphism $X \xrightarrow[]{} Y$ between weakly contractible objects is a weak equivalence
\item[(Cof-Mon)] For every cofibration $f : X \xhookrightarrow{} Y$ and for every object $Z$, the induced morphisms $id_Z \times f : Z \times X \xhookrightarrow{} Z \times Y$ $f \times id_Z : X \times Z \xhookrightarrow{} Y \times Z$ are cofibrations.
\item[(Sub)] For every pushout diagram of the form
    \[\begin{tikzcd}[ampersand replacement=\&]
	\& W \& \\
	X \&\& Y \\
	\& Z
	\arrow[hook, from=2-1, to=1-2]
	\arrow[from=2-3, to=1-2]
	\arrow[from=3-2, to=2-1]
	\arrow[hook, from=3-2, to=2-3]
\end{tikzcd}\]
in $\mathcal{C}$, if $X,Y,Z$ are objects in $\mathcal{C}^{\times}$, then so is $W$.
\item[(Weq-Mon)] $\mathcal{C}^{\times}$ is a monoidal subcategory of $\mathcal{C}$ such that for every weak equivalence $f : X \xrightarrow[]{\sim} Y$ and every object $Z$ in $\mathcal{C}^{\times}$, the induced morphisms $id_Z \times f : Z \times X \xrightarrow{\sim} Z \times Y$ $f \times id_Z : X \times Z \xrightarrow{\sim} Y \times Z$ are weak equivalences as well.
\end{enumerate}
\end{definition}
\begin{theorem}(Theorem 2.7 \cite{dandrea_quantum_2026})
The opposite category of $C^*$-algebras with unital $*$-homomorphisms as opposite morphisms, zero $C^*$-algebra as the initial object, $\mathbb{C}$ as the terminal object, surjective $*$-homomorphisms as cofibrations, $*$-homomorphisms inducing isomorphism on $K$-theory as weak equivalences, $\mathcal{N}_{max}$ as the distinguished subcategory and maximal tensor product is a cw-Waldhausen symmetric monoidal category denoted by $\mathcal{CpctQuant}^{\times}$.
\end{theorem}
\begin{remark}
The category of compact Hausdorff spaces with embeddings as cofibrations and $K$-equivalences as weak equivalences is a cw-Waldhausen category.
\end{remark}
An object $\mathbf{S}^{d-1}$ in $\mathcal{CpctQuant}^{\times}$ is called a $K-sphere$ if it has the same $K$-theory as that of $S^{d-1}$ and an object $\mathbf{B}^d$ in $\mathcal{CpctQuant}^{\times}$ is called a $K-ball$ if it has the same $K$-theory as $B^d$. A cofibration $\partial : \mathbf{S}^{d-1} \xhookrightarrow{} \mathbf{B}^d$ is called a $K$-boundary map if it induces the short exact sequence on the $K$-groups:
\[
\begin{tikzcd}
	0 & {K^0(B^{2d})} & {K^0(S^{2d})} & 0 & 0 \\
	& {\mathbb{Z}} & {\mathbb{Z}}
	\arrow[from=1-1, to=1-2]
	\arrow[from=1-2, to=1-3]
	\arrow["\vcong"{description}, draw=none, from=1-2, to=2-2]
	\arrow[from=1-3, to=1-4]
	\arrow["\vcong"{description}, draw=none, from=1-3, to=2-3]
	\arrow[from=1-4, to=1-5]
\end{tikzcd}
\]

\[\begin{tikzcd}
	0 & {K^1(B^{2d})} & {K^1(S^{2d})} & {\mathbb{Z}} & 0 \\
	& {0} & {\mathbb{Z}}
	\arrow[from=1-1, to=1-2]
	\arrow[from=1-2, to=1-3]
	\arrow["\vcong"{description}, draw=none, from=1-2, to=2-2]
	\arrow[from=1-3, to=1-4]
	\arrow["\vcong"{description}, draw=none, from=1-3, to=2-3]
	\arrow[from=1-4, to=1-5]
\end{tikzcd}\]

and

\[\begin{tikzcd}
	0 & {K^0(B^{2d+1})} & {K^0(S^{2d+1})} & {\mathbb{Z}} & 0 \\
	& {\mathbb{Z}} & {\mathbb{Z} \bigoplus \mathbb{Z}}
	\arrow[from=1-1, to=1-2]
	\arrow[from=1-2, to=1-3]
	\arrow["\vcong"{description}, draw=none, from=1-2, to=2-2]
	\arrow[from=1-3, to=1-4]
	\arrow["\vcong"{description}, draw=none, from=1-3, to=2-3]
	\arrow[from=1-4, to=1-5]
\end{tikzcd}\]

\[\begin{tikzcd}
	0 & {K^1(B^{2d})} & {K^1(S^{2d})} & {0} & 0 \\
	& {0} & {0}
	\arrow[from=1-1, to=1-2]
	\arrow[from=1-2, to=1-3]
	\arrow["\vcong"{description}, draw=none, from=1-2, to=2-2]
	\arrow[from=1-3, to=1-4]
	\arrow["\vcong"{description}, draw=none, from=1-3, to=2-3]
	\arrow[from=1-4, to=1-5]
\end{tikzcd}\]
\begin{definition}(\cite[Definition 2.17]{dandrea_quantum_2026})
A finite quantum K-weak CW-complex is an object $X$ in the category $Ho(\mathcal{CpctQuant}^{\times})$ admitting a finite sequence of \textit{weak cofibrations} of the form

\[\begin{tikzcd}[ampersand replacement=\&]
	\& {\tilde{X^{d_{k}}}} \& \\
	{X^{d_{k-1}}} \&\& {X^{d_{k}}}
	\arrow[from=2-1, to=1-2]
	\arrow["\sim"', from=2-3, to=1-2]
\end{tikzcd}\]
such that $X^{d_0}$ is finite and the sequence is compatible with pushouts in $\mathcal{CpctQuant}^{\times}$ i.e.
\[\begin{tikzcd}
	& {\tilde{X^{d_k}}} & \\
	{X^{d_{k-1}}} && {\bigsqcup_{i=1}^{i_k} B_i^{d_k}} \\
	& {\bigsqcup_{i=1}^{i_k} S_i^{d_k}}
	\arrow[from=2-1, to=1-2]
	\arrow[from=2-3, to=1-2]
	\arrow[from=3-2, to=2-1]
	\arrow["\partial"', hook, from=3-2, to=2-3]
\end{tikzcd}\]
\end{definition}

\section{Quantum Real Projective spaces}
Classically, a real projective space $\mathbb{R}P^n$ is the topological space of lines passing through origin in $\mathbb{R}^{n+1}$. It has the topology of $\mathbb{R}^{n+1}$ with the antipodal points identified.  
Quantum real projective plane was first introduced algebraically ($\mathcal{O}(\mathbb{R}P_q^2$)) by Hajac in \cite{hajac_strong_1996} and subsequently the $C^*$-algebra was studied in detail in \cite{hajac_quantum_2003}. This was generalised to higher dimensions by \cite{hong_quantum_2002} as fixed point subalgebras of the $C^*$-algebras of quantum spheres under a $\mathbb{Z}_2$ action. As a consequence, they also have a presentation as the graph $C^*$-algebra.

The following result is due to \cite{hong_quantum_2002}. 
\begin{proposition}(\cite[Section 4.2, Proposition 5.2]{hong_quantum_2002})
The graph corresponding to the odd-dimensional quantum real projective space $C(\mathbb{R}P_q^{2n+1})$ is given by $E= (\{v_1,v_2,\dots,v_{n+1}\},\{e_{ij,jk} \ \vert \ 1\leq i \leq j \leq k \leq n+1 \}, r,s )$ where, $r(e_{ij,jk}) = v_k$ and $s(e_{ij,jk}) = v_i$, while the even-dimensional quantum real projective space $C(\mathbb{R}P_q^{2n})$ is given by $E= (\{v_1,v_2,\dots,v_{n+1}\},\{e_{ij,jk} \ \vert \ 1\leq i \leq j \leq k \leq n+1 \text{ and } i \leq n \}, r,s )$ where, $r(e_{ij,jk}) = v_k$ and $s(e_{ij,jk}) = v_i$.
\label{rpga}
\end{proposition}

Equivalently, the odd-dimensional quantum real projective spaces also admit a topological graph presentation which will be crucial in constructing a quantum CW-complex for them. 
\begin{lemma}
The topological graph corresponding to the quantum odd real projective spaces $C(\mathbb{R}P_q^{2n+1})$ given by :
$E= (\{v_1, v_2, \dots, v_{n}\} \sqcup S^1, \{e_{ij,jk} \  \vert \ 1 \leq i \leq j \leq k \leq n \} \sqcup_{1\leq i\leq n, i\leq j < n+1} X_{ij,jn+1},r,s) $ with $X_{ij,jn+1} \cong S^1$ and $r(e_{ij,jk}) = v_k, s(e_{ij,jk}) = v_i , r|_{X_{ij,jn+1}} = id_{S^1}, s(X_{ij,jn+1}) = v_i$.
\label{rp3ta}
\end{lemma}
\begin{proof}
Denote $B = C^*(F)$, where $F$ is the directed graph corresponding to $C(\mathbb{R}P_q^{2n+1})$ as given in Proposition \ref{rpga} and let $z_{X_{ij,jn+1}}$ be the unitary generator of $X_{ij,jn+1}$. Now consider the two maps:
$\psi^0 : C_0(E^0) \xrightarrow[]{} B$ and $\psi^1 : X_E \xrightarrow[]{} B$
\begin{align*}
   \psi^1((\delta_{e_{ij,jk}})) &\mapsto S_{e_{ij,jk}} \\
   \psi^1((z_{X_{ij,jn+1}})) &\mapsto S_{e_{ij,jn+1}} \\
   \psi^0((\delta_{v_{i}})) &\mapsto P_{v_i}  \\
   \psi^0(0,z) &\mapsto S_{e_{n+1n+1,n+1n+1}} \\
   \psi^0(0,1) &\mapsto P_{v_{n+1}} \\
\end{align*}
Observe that $\langle z_{X_{ij,jn+1}}, z_{X_{ij,jn+1}} \rangle_{C_0(E^0)} = (0,1_{S^1}) \in C_0(\{v_i\}_{i=1}^{n} \sqcup S^1)$, due to which it is routine to check that the conditions (1) and (2) for $(\psi^0, \psi^1)$ to be a representation hold. Furthermore, $E^0_{rg} = \{v_1,v_2,\dots, v_n\}$, so we need to check the covariance condition at $v_i$. Firstly, $\varphi_X(\delta_{v_i}) = id_{C_0(s^{-1}(v_i)}$ so,
\begin{align*}
    \psi^{(1)}(id_{C_0(s^{-1}(v_i))}) &= \psi^{(1)}(\sum \theta_{\delta_{e_{ij,jk}},\delta_{e_{ij,jk}}} + \sum \theta_{z_{X_{ij,jn+1}},z_{X_{ij,jn+1}}}) \\
                                 &= \sum S_{e_{ij,jk}}S_{e_{ij,jk}}^* + \sum S_{e_{ij,jn+1}} S_{e_{ij,jn+1}}^* \\
                                 &= P_{v_i} \\
                                 &= \psi^0(\delta_{v_i})
\end{align*}

So we have that $(\psi^0, \psi^1)$ is a covariant representation of the given topological graph. So by universality, we have a $*$-homomorphism $\mu : C^*(E) \xrightarrow[]{} C^*(F)$.

Furthermore, the topological graph $E$ is topologically free as it has no base points of loops without exits. So $\mu$ is injective when it is restricted to $C_0(E^0)$ and the graph is topologically free, so by Cuntz-Krieger uniqueness theorem for topological graphs, we have that $\mu$ is injective everywhere. Furthermore, as all generators of the codomain are in the image of $\mu$, it is also surjective and we are done. 
\end{proof}

The author explored topological graphs and their gluings in a recent paper \cite{gothe_gluing_2026}. The $C^*$-algebras of q-deformed quantum spheres ($C(S_q^{n}))$ and Hong-Szymański balls ($C(B_q^{n})$) can be reformulated using quantum double suspension. We refer the reader to \cite{hong_quantum_2002}\cite{hong_noncommutative_2008}\cite{gothe_gluing_2026} for a more thorough discussion. Due to \cite[Section 5.2.1]{gothe_gluing_2026}, we have the following characterisation of the Vaksman-Soibelman spheres ($C(S_q^{2n+1})$) and odd-dimensional Hong-Szymański balls ($C(B_q^{2n+1})$) as topological graphs.

The quantum double suspension for a unital $C^*$-algebra $A$ was defined by Hong-Szymański in \cite[Definition 6.1]{hong_quantum_2002} as the extension of $C(S^1)$ by $A \otimes \mathcal{K}$, satisfying certain relations. However, we use the following equivalent definition, due \cite[Definition 4.1]{chakraborty_quantum_2011}.
\begin{definition}(\cite[Definition 6.1]{hong_quantum_2002},\cite[Proposition 2.1]{hong_noncommutative_2008},  \cite[Definition 4.1]{chakraborty_quantum_2011})
 Let $A$ be a unital $C^*$-algebra. Then the quantum double suspension $\Sigma^2A$ of $A$, is the $C^*$-subalgebra of $A \otimes \mathcal{T}$ generated by the elements $a \otimes p$ and $1 \otimes s$, where $s$ is the isometry generating $\mathcal{T}$ and $p = 1-ss^*$.
\end{definition}
\begin{remark}
By definition we have that $C(B_q^{2n+2}) = \Sigma^2C(B_q^n)$ and $C(S_q^{2n+2}) = \Sigma^2C(S_q^n)$ \cite{hong_noncommutative_2008}.
\end{remark}

\begin{proposition}
 The topological graph for $C(S_q^{2n+1})$ is given by $(\{v_1,\dots,v_n \} \sqcup S^1, \{e_{ij} | 1\leq i \leq j \leq n \} \sqcup_{i=1}^{n} X_{in+1} , r, s)$, with $X_{in+1} \cong S^1$ and $r\vert_{X_{in+1}} = id_{S^1}$, $r(e_{ij}) = v_j$ and $s(e_{ij}) = s(X_{in+1}) = v_i$, while for $C(B_q^{2n+1})$, it is given by the quadruple $(\{v_1,\dots,v_n \} \sqcup [-1,1], \{e_{ij} | 1\leq i \leq j \leq n \} \sqcup_{i=1}^{n} Y_{in+1} , r, s)$, with $Y_{in+1} \cong [-1,1]$ and $r\vert_{Y_{in+1}} = id_{[-1,1]}$, $r(e_{ij}) = v_j$ and $s(e_{ij}) = s(Y_{in+1}) = v_i$
\end{proposition}

A visualisation of the topological graphs corresponding to $C(\mathbb{R}P_q^{3})$, $C(B_q^3)$ and $C(S_q^3)$ is given below, the topological graphs corresponding to the higher dimension can be visualised analogously.
\sbox{\boxB}{
\begin{tikzpicture}[scale=0.5,tight/.style={inner sep=1cm}, loose/.style={inner sep=1cm}]
\draw [ultra thick] (0,0)coordinate(O) circle [radius=2cm] ;
\draw[->=.5] (4,0) -- node[midway,above] {2} (2,0) coordinate (X) ;
\draw[->=.5] (4,0)  [bend right=15] node[midway,above] {2}  to (-2,0) ;
\draw[->=.5] (4,0)  [bend right]  to (0,2) ;
\draw[->=.5] (4,0)  [bend left]  to (0,-2) ;
\draw[->=.5] (4,0) edge [out=120, in=60, distance=2cm,decorate] (4,0) ;
\draw[->=.5] (4,0) edge [out=120, in=60, distance=2cm] (4,0);
\draw[->=.5] (4,0) edge [tight,decorate]  node[midway,below] {2} (-1.414,-1.414);
\draw[->=.5] (4,0) edge [tight] node[midway, below] {2} (-1.414,-1.414) ;
\draw[->=.5] (4,0) edge [tight,
                    ] node[midway,below] {2} (-1.414,1.414) ;
\draw[->=.5] (4,0) edge [tight, decorate
                    ] node[midway,below] {2} (-1.414,1.414) ;

\end{tikzpicture}
}
\begin{figure}[htp]
    \centering
\usebox{\boxB}
\caption{$C(\mathbb{R}P_q^3)$}
\end{figure}
\clearpage
Here, the labelling "2" on the edges signifies that there are 2 copies of the edges i.e $S^1 \sqcup S^1$.
\sbox{\boxC}{\begin{tikzpicture}[scale=0.5,tight/.style={inner sep=1cm}, loose/.style={inner sep=1cm}]
\draw (0,-2)--(0,2)  [ultra thick] ;
\draw[->=.5] (4,0) -- (0,2) coordinate (X) node[midway,below] {};
\draw[->=.5] (4,0) -- (0,-2) coordinate (X) node[midway,below] {};
\draw[->=.5] (4,0) -- (0,1) coordinate (X) node[midway,below] {};
\draw[->=.5] (4,0) -- (0,-1) coordinate (X) node[midway,below] {};
\draw[->=.5] (4,0) edge [out=120, in=60, distance=2cm] (4,0);
\draw[->=.5] (4,0) edge [out=120, in=60, distance=2cm,decorate] (4,0);
\end{tikzpicture}}
\begin{figure}[htp]
\centering
\usebox{\boxC}
\caption{$C(B_q^3)$}
\end{figure}
\begin{figure}[!h]
    \centering
\begin{tikzpicture}[scale=0.5,tight/.style={inner sep=1cm}, loose/.style={inner sep=1cm}]
\draw [ultra thick] (0,0)coordinate(O) circle [radius=2cm] ;
\draw[->=.5] (4,0) --  (2,0) coordinate (X) ;
\draw[->=.5] (4,0)  [bend right=15]  to (-2,0) ;
\draw[->=.5] (4,0)  [bend right]  to (0,2) ;
\draw[->=.5] (4,0)  [bend left]  to (0,-2) ;
\draw[->=.5] (4,0) edge [out=120, in=60, distance=2cm,decorate] (4,0) ;
\draw[->=.5] (4,0) edge [out=120, in=60, distance=2cm] (4,0);
\draw[->=.5] (4,0) edge [tight,decorate]  (-1.414,-1.414);
\draw[->=.5] (4,0) edge [tight]  (-1.414,-1.414) ;
\draw[->=.5] (4,0) edge [tight,
                    ]  (-1.414,1.414) ;
\draw[->=.5] (4,0) edge [tight, decorate
                    ]  (-1.414,1.414) ;
\end{tikzpicture}
\caption{$C(S_q^3)$}
\end{figure}
\begin{remark}
For brevity we adopt the following convention that $f_{ij,jk}$ denotes an element in $C(X_{ij,jk})$.
\end{remark}
Consider the following commutative diagram:
    \[\begin{tikzcd}[ampersand replacement=\&]
	\& {C(\mathbb{R}P_q^{2n})} \\
	{C(\mathbb{R}P_q^{2n-1})} \&\& {C(B_q^{2n})} \\
	\& {C(S_q^{2n-1})}
	\arrow["\gamma"', two heads, from=1-2, to=2-1]
	\arrow["\delta", hook', from=1-2, to=2-3]
	\arrow["\beta"', hook', from=2-1, to=3-2]
	\arrow["\alpha", two heads, from=2-3, to=3-2]
\label{realeven}
\end{tikzcd}\]

\begin{proposition}
 The commutative diagram \ref{realeven} is a pullback diagram.
\end{proposition}
\begin{proof}
This pullback diagram is a special case of the covariant functoriality of graph algebras proven by \cite{hajac_covariant_2024} (See Example 5.3).
\end{proof}

Now we tackle the harder issue of the pullback for odd-dimensional quantum real projective spaces. The primary difference between the odd-dimensional and the even-dimensional cases is the fact that the even-dimensional quantum balls admit a graph algebra presentation, which the odd-dimensional balls do not possess and one needs to go to the topological graphs to get a better idea of this gluing.

\begin{proposition}(\cite[Proposition 5.1]{hong_quantum_2002})
The graph of $C(S_q^{2n})$ consists of vertices $\{v_i\}_{i=1}^n \sqcup \{w_1,w_2\}$ and edges $\{e_{ij} \vert 1 \leq i \leq j  < n\} \sqcup \{g_{i,k} \vert 1 \leq i \leq n, k = 1,2\}$ with $s(e_{ij}) = s(g_{ik}) = v_i$ and $r(e_{ij}) = v_j, r(g_{ik}) = w_k$. 
\end{proposition}
\begin{proposition}
There is a exact sequence of $C^*$-algebras:
\[\begin{tikzcd}
	0 & {\mathcal{K}\otimes C_0(-1,1)} & {C(B_q^{2n+1})} & {C(S_q^{2n})} & 0
	\arrow[from=1-1, to=1-2]
	\arrow[hook, from=1-2, to=1-3]
	\arrow[two heads, from=1-3, to=1-4]
	\arrow[from=1-4, to=1-5]
\end{tikzcd}\]
More explicitly, $ker\alpha = \langle t^0(f) \vert f \in C_0(-1,1) \rangle$.
\label{hsker}
\end{proposition}

\begin{proof}
The $*$-homomorphism $\alpha$ is induced by the factor map $m^0 : \{v_i\} \sqcup \{-1,1\} \xrightarrow[]{} \{v_i\} \sqcup [-1,1]$ and $m^1 : \{e_{ij,jk} \} \sqcup \{-1,1\}_{in+1} \xrightarrow[]{} \{e_{ij,jk}\} \sqcup Y_{in+1}$. Note that here, we rewrite $\{w_1,w_2\}$ as $\{-1,1\}$ and $g_{i,k}$ as $\{-1,1\}_{in+1}$ as to emphasise the embedding of the final vertices and edges inside the closed interval. So by \cite[Proposition 4.1]{gothe_gluing_2026} we have that the admissible pair corresponding to the (gauge-invariant) ideal $ker\alpha$ is $\{v_i\} \sqcup \{-1,1\}, \{-1,1\}$.

Consider the subset $V = (-1,1) \subseteq E^0$. This is a hereditary and saturated subset of $E^0$ (in the sense of Definition \ref{hss}), so we have that $S(H(V)) = V$. By \cite[Proposition 6.1]{katsura_class_2006-1}, we have that for the ideal generated by $t^0(C_0(V))$ is gauge-invariant, and the admissible pair is given by $(E^0 \backslash S(H(V)) , E^0_{sg}\backslash S(H(V))$, which in this case gives the admissible pair $(\{v_1,v_2,\dots,v_n\} \sqcup \{-1,1\} , \{-1,1\}$. So by the bijective correspondence between gauge-invariant ideals and admissible pairs, we have that $ker\alpha = \langle t^0(f) | f\in C_0(-1,1) \rangle$.
\end{proof}

\begin{proposition}(\cite[Proposition 5.2]{hong_quantum_2002})
The map $\beta: C(\mathbb{R}P_q^{2n}) \xhookrightarrow{} C(S_q^{2n})$ is given by:
\begin{align*}
    P_{v_i} &\mapsto P_{v_i} \hspace{2cm} i < n+1 \\
    P_{v_{n+1}} &\mapsto P_{w_1} + P_{w_2} \\
    S_{e_{ij,jk}} &\mapsto S_{e_{ij}}S_{e_{jk}} \hspace{2cm} k \neq n+1 \\
    S_{e_{ij,jn+1}} &\mapsto S_{e_{i,i+k-1}} (S_{g_{i+k-1,1}} + S_{g_{i+k-1,2}}) \\
    S_{e_{ii,in+1}} &\mapsto S_{g_{i,1}} - S_{g_{i,-1}}
\end{align*}
\end{proposition}
\begin{theorem}
The following commutative diagram is a pullback diagram, yielding the quantum CW-complex structure for quantum real projective spaces $\mathbb{C} \twoheadleftarrow{} C(\mathbb{R}P_q^1) \twoheadleftarrow{} C(\mathbb{R}P_q^2) \twoheadleftarrow{} \dots \twoheadleftarrow{} C(\mathbb{R}P_q^n)$
    \[\begin{tikzcd}[ampersand replacement=\&]
	\& {C(\mathbb{R}P_q^{2n+1})} \\
	{C(\mathbb{R}P_q^{2n})} \&\& {C(B_q^{2n+1})} \\
	\& {C(S_q^{2n})}
	\arrow["\gamma"', two heads, from=1-2, to=2-1]
	\arrow["\delta", hook', from=1-2, to=2-3]
	\arrow["\beta"', hook', from=2-1, to=3-2]
	\arrow["\alpha", two heads, from=2-3, to=3-2]
\end{tikzcd}\]
The $*$-homomorphisms $\gamma$ and $\delta$ are given in the following lemmas.
\end{theorem}

\begin{lemma}
The map $\delta : C(\mathbb{R}P_q^{2n+1}) \xrightarrow[]{} C(B_q^{2n+1})$ given by:

\begin{align*}
   t^1(\delta_{e_{ij,jk}}) &\mapsto t^1((\delta_{e_{ij}}))t^1((\delta_{e_{jk}})) \\
   t^1(f_{ij,jn+1}) &\mapsto t^1(\delta_{e_{ij}})t^1(f_{jn+1} \circ q)  \\
   t^1(f_{i n+1,n+1 n+1}) &\mapsto t^1(f_{in+1} \circ q) \\
   t^0(\delta_{v_{i}}, 0) &\mapsto t^0(\delta_{v_i}, 0)  \\
   t^0(0,f) &\mapsto t^0(0,f \circ q)
\end{align*}
is injective. Here we use the fact that, $q : [-1,1] \xrightarrow[]{} S^1 $ is continuous surjective map by identifying the end points induces an injective embedding $C(S^1) \xhookrightarrow{} C[-1,1]$. 
\end{lemma}
\begin{proof}
Consider the tuple of maps $(\psi^0, \psi^1) : (C_0(E^0), X_E) \xrightarrow[]{} C(B_q^{2n+1})$ defined by:

\begin{align*}
    \psi^1((\delta_{e_{ij,jk}})) &= t^1((\delta_{e_{ij}}))t^1((\delta_{e_{jk}})) \\
   \psi^1(f_{ij,jn+1}) &= t^1(\delta_{e_{ij}})t^1(f_{jn+1} \circ q)  \\
   \psi^1(f_{i n+1,n+1 n+1}) &= t^1(f_{in+1} \circ q) \\
   \psi^0((\delta_{v_{i}})) &= t^0(\delta_{v_i})  \\
   \psi^0(0,f) &= t^0(0,f \circ q)
\end{align*}

$(\psi^0, \psi^1)$ is a covariant representation of $C(\mathbb{R}P_q^{2n+1})$. So we have an induced $*$-homomorphism $\delta$ as defined above.

Now, $C(\mathbb{R}P_q^{2n+1})$ is topologically free and $\delta |_{C_0(E^0)}$ is injective, so by Cuntz-Krieger uniqueness theorem, $\delta$ is injective everywhere.
\end{proof}
\begin{lemma}
Let $E$ be the directed graph corresponding to $C(\mathbb{R}P_q^{2n})$ as given in Proposition \ref{rpga} and $F$ be the topological graph corresponding to $C(\mathbb{R}P_q^{2n+1})$ as in Lemma \ref{rp3ta}. Then the maps $m^0 : E^0 \xrightarrow[]{} F^0$ and $m^1 : E^1 \xrightarrow[]{} F^1$, be given by:
\begin{align*}
    m^0(v_i) &= v_i \hspace{3cm} i \neq n+1 \\
    m^0(v_{n+1}) &= \{1\} (\in S^1) \\
    m^1(e_{ij,jk}) &= e_{ij,jk} \\
    m^1(e_{ij,jn+1}) &= \{1\} (\in X_{ij,jn+1}) \\
\end{align*}
gives a injective regular factor map, and consequently we have a surjective $*$-homomorphism $\gamma: C(\mathbb{R}P_q^{2n+1}) \xrightarrow[]{} C(\mathbb{R}P_q^{2n})$, with $ker\gamma = \langle t^0(f) \vert f \in C_0(S^1\backslash\{1\})\rangle$.

\end{lemma} 
\begin{proof}
 It is easy to check that $(m^0,m^1)$ is a proper factor map, combined with the fact that $E^0_{rg} = F^0_{rg} = \{v_1, v_2, \dots ,v_n\}$, makes it a regular factor map. Now by \cite[Lemma 4.1 ]{gothe_gluing_2026}, the invariant subset corresponding to the ideal $ker\gamma$ is $\{v_1,v_2,\dots,v_n,1\}$.
    
$V = S^1 \backslash \{1\}$ is a saturated and hereditary subset of $E^0$. And by the same argument of Proposition \ref{hsker}, the result follows. 
\end{proof}    
\begin{proof}
Proof of the theorem

    The Pedersen conditions are:
    \begin{enumerate}
        \item $ker\delta \cap ker\gamma = 0$
        
        $\delta$ is injective.
        
        \item $\beta^{-1}(\alpha(C(B_q^{2n+1}))) = \gamma(C(\mathbb{R}P_q^{2n+1})$
        
        $\gamma$ and $\alpha$ are surjective.
        
        \item $\delta(ker \gamma) = ker\alpha$
        
        Let, $a \in ker\alpha$, then say $a = t^n(\xi)t^0(f)t^m(\eta)^*$. If $\xi = \xi_1\otimes \xi_2 \otimes \dots \otimes \xi_n$, then,  $a = t^1(\xi_1)t^1(\xi_2)\dots t^1(\xi_n) t^0(f) t^m(\eta)^*$. But $t^1(\xi_n)t^0(f) = t^1(\xi_n \cdot f)$ and as $f \in C_0(-1,1)$ and $r = id$, we have that $\xi_n \in C_0(-1,1)$ and by construction of $\delta$, it is routine computation to check that image of $\delta$ will contain all paths ending in $(-1,1)$. Similar result holds for $t^m(\eta)^*$, and as we have that the generator of $ker\gamma$ maps to generator of $ker \alpha$, we have that $ker \alpha \subseteq \delta(ker\gamma)$ and we are done. The other subset inequality follows from the commutativity of the diagram.
    \end{enumerate}
 
\end{proof}
 
\section{Multipullbacks of \texorpdfstring{$C^*$}{C*}-algebras}

Multipullbacks of $C^*$-algebras were first defined in \cite{pedersen_pullback_1999} (called \textit{Multirestricted direct sums}). These were used to construct the odd-dimensional multipullback quantum spheres and complex projective spaces \cite{hajac_k-theory_2018}. These $C^*$-algebras are the boundaries of a version of even-dimensional quantum balls given by n-copies of the Toeplitz algebra. We refer to \cite{hajac_k-theory_2018} for a general introduction of multipullbacks of $C^*$-algebras and the odd-dimensional multipullback quantum spheres and the complex projective spaces.

\begin{definition}(\cite{pedersen_pullback_1999})
Let $I$ be a finite indexing set and $\pi_j^i : A_i \xrightarrow[]{} A_{ij} \ \  i,j \in I$ and $A_{ij} = A_{ji}$ be a set of $*$-homomorphisms between $C^*$-algebras. Then the multipullback $C^*$-algebra $\oplus_{ij}A_i$ is the categorical wide pullback of the diagram defined by the $*$-homomorphisms. Alternatively, $\oplus_{ij}A_i$ can be defined as the $C^*$-algebra : $\oplus_{ij}A_i = \{ (a_1,\dots,a_i,\dots,a_j,\dots) \in \oplus A_i \ \vert \ \pi_j^i(a_i) = \pi_i^j(a_j) \}$
\end{definition}
The multipullback odd spheres can be regarded as the boundary of the $2n+2$-ball given by $\mathcal{T}^{n+1}$. To fix notation, let $t_i = \underbrace{1 \otimes 1\otimes \dots 1}_{i \text{ times }}\otimes s \otimes \underbrace{1 \otimes \dots \otimes 1}_{n-i \text{ times }}$ where, $0 \leq i \leq n$.
\begin{definition}(\cite[Definition 3.2]{hajac_k-theory_2018})
 Let $I_i = \langle 1-t_it_i^*\rangle$ and let $I_{ij} = I_i + I_j$. Finally, let the quotient algebra be denoted by $B_i = \mathcal{T}^{n+1} / I_i$ and $B_{ij} = \mathcal{T}^{n+1} / I_{ij} $
 \begin{align*}
     \sigma_i &:  \mathcal{T}^{n+1} \twoheadrightarrow B_i \\
     \pi_j^i &: B_i \twoheadrightarrow B_{ij} 
 \end{align*}
 then the \textit{(untwisted) multipullback odd sphere} is defined as the multipullback $C^*$-algebra:
 \[
 C(S_H^{2n+1}) := \{ (b_0, b_1, \dots, b_n) \in \bigoplus B_i \vert \pi_j^i(b_i) = \pi_i^j(b_j) \text{ for all } 0 \leq i<j \leq n \} 
 \]
\end{definition}
\begin{proposition}(\cite[Theorem 3.3]{hajac_k-theory_2018})
The multipullback odd sphere $C^*$-algebra $C^*(S_H^{2n+1})$ is generated by a set of generators - $\{s_0,s_1,\dots,s_n\}$ such that the following conditions hold:

\begin{enumerate}
    \item $s_i^*s_i = 1$ \hspace{3cm} $i \in \{0,1,2,\dots,n\}$
    \item $s_is_j = s_js_i$  \hspace{3cm} $ i \neq j$
    \item $s_is_j^* = s_j^*s_i$ \hspace{3cm} $i < j$
    \item $\prod_i (1-s_is_i^*) = 0$
\end{enumerate}
\end{proposition}

\begin{proposition}(\cite[Lemma 5.1]{hajac_k-theory_2018})
The following is a short exact sequence:
\[\begin{tikzcd}[ampersand replacement=\&]
	0 \& {\mathcal{K}} \& {\mathcal{T}^n} \& {C(S_H^{2n-1})} \& 0
	\arrow[from=1-1, to=1-2]
	\arrow[hook, from=1-2, to=1-3]
	\arrow["{\sigma_n}", two heads, from=1-3, to=1-4]
	\arrow[from=1-4, to=1-5]
\end{tikzcd}\]
where $\sigma_n(t_i) = s_i$.
\label{bdos}
\end{proposition}

So analogously we define the multipullback even spheres as the boundary of an odd-dimensional ball.

Let $I$ be the ideal generated by $\prod (1-t_it_i^*)(1-z^2)$.
 \begin{definition}
 Let $C(S_H^{2n})$ be the quotient $C^*$-algebra $\mathcal{T}^n \otimes C[-1,1] / I$, called the even-dimensional multipullback quantum spheres.
 \end{definition}
 We will justify this name by proving this $C^*$-algebra is isomorphic to a multipullback of $C^*$-algebras akin to the construction of a classical construction of a cylinder by the means of attaching balls ($B^{2n}$) to the ends of an open cylinder ($S^{2n-1} \times [-1.1]$).
\begin{proposition}
 $C^*(s_0,s_1,\dots,s_n)$ be the $C^*$-algebra generated by a set of elements $s_0, s_1 , s_2, \dots , s_n$ satisfying the relations:
\begin{enumerate}
    \item $s_i s_j = s_j s_i$
    \item $s_i^* s_j = s_j s_i^*$ for $ i < j$
    \item $s_i^*s_i = 1$ for $ i < n$
    \item $\prod (1-s_is_i^*) = 0$
    \item $s_n = s_n^*$ and $||s_n|| \leq 1$
\end{enumerate}
then we have that $C(S_H^{2n}) \cong C^*(s_0,s_1,\dots,s_n)$.
\end{proposition}
\begin{proof}
Let $\sigma : \mathcal{T}^n \otimes C[-1,1] \xrightarrow[]{} C^*(s_0,s_1,\dots,s_n)$ be the following map:
\begin{align*}
    \sigma(t_i) &= s_{i} \\
    \sigma(z) &= s_n
\end{align*}

This is clearly a $^*$-homomorphism, and the kernel is the ideal generated by the element $\prod (1-t_it_i^*)(1-z^2)$, which is isomorphic to $\mathcal{K} \otimes C_0(-1,1)$ so we have the required isomorphism between $C(S_H^{2n})$ and $C^*(s_0,s_1,\dots,s_n)$.
\end{proof}

\begin{proposition}
Let the following be 1-surjective pullback of unital $C^*$-algebras. 
\[\begin{tikzcd}[ampersand replacement=\&]
	\& X \& \\
	A \&\& B \\
	\& C
	\arrow["\delta", two heads, from=1-2, to=2-1]
	\arrow["\gamma"', from=1-2, to=2-3]
	\arrow["\alpha", from=2-1, to=3-2]
	\arrow["\beta"', two heads, from=2-3, to=3-2]
\end{tikzcd}\]
Then the induced diagram on the quantum double suspensions is also a pullback diagram, viz, $\Sigma^2 X \cong \Sigma^2A \bigoplus_{\Sigma^2C} \Sigma^2B$.
\label{qdsp}
\end{proposition}
\begin{proof}
    If $\mu : A \xrightarrow[]{} B$ is a unital $*$-homomorphism between unital $C^*$-algebras, then the induced map between the quantum double suspensions is given by $\Sigma^2\mu : \Sigma^2A \xrightarrow[]{} \Sigma^2B$: $\Sigma^2\mu(a \otimes p) = \mu(a) \otimes p$ and $\Sigma^2 (1 \otimes s) = 1 \otimes s$. Due to \cite{saurabh_cso_q2n1so_q2n-1_2026}, $ker \Sigma^2 \mu = ker \mu \otimes \mathcal{K}$. $\mathcal{T}$ is an exact $C^*$-algebra, so tensoring with it preserves pullbacks \cite{pedersen_pullback_1999}. The result follows from simple checking the *-subalgebra defining the quantum double suspension also preserves the pullback.
\end{proof}
\begin{lemma} 
We have the following pullback diagram

\[\begin{tikzcd}[ampersand replacement=\&]
	{C(B_q^{2n+1})} \& {\mathcal{T}^n\otimes C[-1,1]} \\
	{C(S_q^{2n})} \& {C(S_H^{2n})}
	\arrow[hook, from=1-1, to=1-2]
	\arrow[two heads, from=1-1, to=2-1]
	\arrow[two heads, from=1-2, to=2-2]
	\arrow[hook, from=2-1, to=2-2]
\end{tikzcd}\]
\end{lemma}

\begin{proof}
For n = 0, this is clearly a pullback diagram as $C(B_q^1) = C[-1,1]$.

For higher dimension we proceed using induction. Consider the following diagram:

\[\begin{tikzcd}[ampersand replacement=\&]
	{C(B_q^{2n+1})} \& {\Sigma^2(\mathcal{T}^{n-1}\otimes C[-1,1])} \& {\mathcal{T}^n\otimes C[-1,1]} \\
	{C(S_q^{2n})} \& {\Sigma^2(C(S_H^{2n-2}))} \& {C(S_H^{2n})}
	\arrow[hook, from=1-1, to=1-2]
	\arrow[two heads, from=1-1, to=2-1]
	\arrow[hook, from=1-2, to=1-3]
	\arrow[two heads, from=1-2, to=2-2]
	\arrow[two heads, from=1-3, to=2-3]
	\arrow[hook, from=2-1, to=2-2]
	\arrow[hook, from=2-2, to=2-3]
\end{tikzcd}\]

The induction hypothesis is that the following diagram is a pullback:

\[\begin{tikzcd}[ampersand replacement=\&]
	{C(B_q^{2n-1})} \& {\mathcal{T}^{n-1}\otimes C[-1,1]} \\
	{C(S_q^{2n-2})} \& {C(S_H^{2n-2})}
	\arrow[hook, from=1-1, to=1-2]
	\arrow[two heads, from=1-1, to=2-1]
	\arrow[two heads, from=1-2, to=2-2]
	\arrow[hook, from=2-1, to=2-2]
\end{tikzcd}\]

So by Proposition \ref{qdsp}, the following is also a pullback:

\[\begin{tikzcd}[ampersand replacement=\&]
	{\Sigma^2C(B_q^{2n-1})} \& {\Sigma^2(\mathcal{T}^{n-1}\otimes C[-1,1])} \\
	{\Sigma^2C(S_q^{2n-2})} \& {\Sigma^2(C(S_H^{2n-2}))}
	\arrow[hook, from=1-1, to=1-2]
	\arrow[two heads, from=1-1, to=2-1]
	\arrow[two heads, from=1-2, to=2-2]
	\arrow[hook, from=2-1, to=2-2]
\end{tikzcd}\]
which gives us that the left square is a pullback.

We know that the following is a short exact sequence:

\[\begin{tikzcd}[ampersand replacement=\&]
	0 \& {\mathcal{K}\otimes C_0(-1,1)} \& {\mathcal{T}^{n-1}\otimes C[-1,1]} \& {C(S_H^{2n-2})} \& 0
	\arrow[from=1-1, to=1-2]
	\arrow[hook, from=1-2, to=1-3]
	\arrow["\alpha", two heads, from=1-3, to=1-4]
	\arrow[from=1-4, to=1-5]
\end{tikzcd}\]

So the induced map on the quantum double suspensions is $\Sigma^2\alpha : \Sigma^2(\mathcal{T}^{n-1} \otimes C[-1,1] \xrightarrow[]{} \Sigma^2C(S_H^{2n-2})$ is $\Sigma^2\alpha(t_0 \otimes 1) = t_0 \otimes 1$ and $\Sigma^2\alpha((1-t_0t_0^*) \otimes a) = (1-t_0t_0^*) \otimes \alpha(a)$. By \cite{saurabh_cso_q2n1so_q2n-1_2026} $ker \Sigma^2 \alpha = \mathcal{K} \otimes ker\alpha$. Clearly, we have that $\iota(ker\alpha) = ker\Sigma^2\alpha$, so the right square is pullback, consequently, by pasting lemma, the outer square is also a pullback.
\end{proof}
\begin{proposition}
 
The inclusion $\iota : C(B_q^{2n+1}) \xhookrightarrow{} \mathcal{T}^n \otimes C[-1,1]$ induces an isomorphism on K-theory.
\end{proposition} 

\begin{proof}

The $K$-theory of $\mathcal{T}$ is free and abelian, so by Kunneth formula for $C^*$-algebras \cite[Theorem 2.14]{schochet_topological_1982}, we have the isomorphism $K_{*}(\mathcal{T} \otimes C[-1,1]) \cong K_{*}(\mathcal{T}) \otimes_{\mathbb{Z}} K_*(C[-1,1])$ and consequently we have $K_0(\mathcal{T}^n \otimes C[-1,1]) \cong \mathbb{Z}$ and $K_1(\mathcal{T}^n \otimes C[-1,1]) \cong \{0\}$ and is generated by class of identity i.e. $[1]_0$. Furthermore, the $K$-theory of $C(B_q^{2n+1})$ is given by $K_0(C(B_q^{2n+1}) \cong \mathbb{Z}$ generated by the class of identity \cite[Proposition 4.1]{hong_noncommutative_2008} and $K_1(C(B_q^{2n+1})) = \{0\}$. The inclusion $C(B_q^{2n+1}) \xhookrightarrow{} \mathcal{T}^n \otimes C[-1,1]$ is a unital $*$-homomorphism, so it induces an isomorphism on $K$-theory.
\end{proof}
\begin{proposition}
The inclusion $ \mu: C(S_q^{2n}) \xhookrightarrow{} C(S_H^{2n})$ also induces an isomorphism on $K$-theory.
\end{proposition}
\begin{proof}
To begin with have the inclusion map $C(B_q^{2n+1}) \xhookrightarrow{} \mathcal{T}^n \otimes C[-1,1]$ and two short exact sequences:

\[\begin{tikzcd}
	0 & {\mathcal{K}\otimes C_0(-1,1)} & {C(B_q^{2n+1})} & {C(S_q^{2n})} & 0 \\
	0 & {\mathcal{K}\otimes C_0(-1,1)} & {\mathcal{T}^n\otimes C[-1,1]} & {C(S_H^{2n})} & 0
	\arrow[from=1-1, to=1-2]
	\arrow[from=1-2, to=1-3]
	\arrow["id", from=1-2, to=2-2]
	\arrow[from=1-3, to=1-4]
	\arrow[hook, from=1-3, to=2-3]
	\arrow[from=1-4, to=1-5]
	\arrow[hook, from=1-4, to=2-4]
	\arrow[from=2-1, to=2-2]
	\arrow[from=2-2, to=2-3]
	\arrow[from=2-3, to=2-4]
	\arrow[from=2-4, to=2-5]
\end{tikzcd}\]

So we have the following induced maps between the six term sequences in $K$-theory

\[\begin{tikzcd}[ampersand replacement=\&,cramped,sep=tiny]
	\&\& {K_0(\mathcal{K}\otimes C_0(-1,1))} \&\& \\
	\&\& {K_0(\mathcal{K}\otimes C_0(-1,1))} \\
	{K_0(\mathcal{T}^n\otimes C[-1,1])} \& {K_0(C(B_q^{2n+1}))} \&\& {K_1(C(S_q^{2n}))} \& {K_1(C(S^{2n}_H)} \\
	{K_0(C(S_H^{2n}))} \& {K_0(C(S_q^{2n}))} \&\& {K_1(C(B_q^{2n+1}))} \& {K_1(\mathcal{T}^n \otimes C[-1,1])} \\
	\&\& {K_1(\mathcal{K}\otimes C_0(-1,1))} \\
	\\
	\&\& {K_1(\mathcal{K} \otimes C_0(-1,1))}
	\arrow[from=1-3, to=3-1]
	\arrow[from=2-3, to=1-3]
	\arrow[from=2-3, to=3-2]
	\arrow[from=3-1, to=4-1]
	\arrow[from=3-2, to=3-1]
	\arrow[from=3-2, to=4-2]
	\arrow[from=3-4, to=2-3]
	\arrow[from=3-4, to=3-5]
	\arrow[from=3-5, to=1-3]
	\arrow[from=4-1, to=7-3]
	\arrow[from=4-2, to=4-1]
	\arrow[from=4-2, to=5-3]
	\arrow[from=4-4, to=3-4]
	\arrow[from=4-4, to=4-5]
	\arrow[from=4-5, to=3-5]
	\arrow[from=5-3, to=4-4]
	\arrow[from=5-3, to=7-3]
	\arrow[from=7-3, to=4-5]
\end{tikzcd}\]
which becomes

\[\begin{tikzcd}
	& 0 &&&& 0 & \\
	{\mathbb{Z}} && 0 && {\mathbb{Z}} && {K_1(C(S_H^{2n}))} \\
	{\mathbb{Z} \oplus \mathbb{Z}} && 0 && {K_0(C(S_H^{2n}))} && 0 \\
	& {\mathbb{Z}} &&&& {\mathbb{Z}}
	\arrow[from=1-2, to=1-6]
	\arrow[from=1-2, to=2-1]
	\arrow[from=1-6, to=2-5]
	\arrow[curve={height=12pt}, from=2-1, to=2-5]
	\arrow[from=2-1, to=3-1]
	\arrow[from=2-3, to=1-2]
	\arrow[curve={height=-12pt}, from=2-3, to=2-7]
	\arrow[from=2-5, to=3-5]
	\arrow[from=2-7, to=1-6]
	\arrow[curve={height=-12pt}, from=3-1, to=3-5]
	\arrow[from=3-1, to=4-2]
	\arrow[from=3-3, to=2-3]
	\arrow[curve={height=12pt}, from=3-3, to=3-7]
	\arrow[from=3-5, to=4-6]
	\arrow[from=3-7, to=2-7]
	\arrow[from=4-2, to=3-3]
	\arrow[from=4-2, to=4-6]
	\arrow[from=4-6, to=3-7]
\end{tikzcd}\]

Clearly, we have that $K_1(C(S_H^{2n})) = 0$, so $\mu_1$ is an isomorphism. Moreover, we have the following:

\[\begin{tikzcd}
	0 & {\mathbb{Z}} & {\mathbb{Z} \oplus \mathbb{Z}} & {\mathbb{Z}} & 0 \\
	0 & {\mathbb{Z}} & {K_0(C(S_H^{2n}))} & {\mathbb{Z}} & 0
	\arrow[from=1-1, to=1-2]
	\arrow[from=1-2, to=1-3]
	\arrow["\cong", from=1-2, to=2-2]
	\arrow[from=1-3, to=1-4]
	\arrow[from=1-3, to=2-3]
	\arrow[from=1-4, to=1-5]
	\arrow["\cong", from=1-4, to=2-4]
	\arrow[from=2-1, to=2-2]
	\arrow[from=2-2, to=2-3]
	\arrow[from=2-3, to=2-4]
	\arrow[from=2-4, to=2-5]
\end{tikzcd}\]

The left and right arrows are isomorphisms, so by Short Five Lemma we have that $\mu_0$ is also an isomorphism.
\end{proof}
\begin{remark}
Note that the last line of the proof also says that $\mathcal{T}^n\otimes C[-1,1]$ is $K$-ball and $C(S^{2n}_H)$ is its boundary $K$-sphere.
\end{remark}
\begin{proposition}
Let $A_i$ and $A_{ij}$ where $i,j \in \{1,2,3\}$ be finite families of $C^*$-algebras with the following $*$-homomorphisms from a $C^*$-algebra $X$: 

$\pi_i : X \xrightarrow[]{} A_i$

$\pi_{j}^i : A_i \xrightarrow[]{} A_{ij}$ 

which form a star-shaped commutative diagram centred at 2 as below:

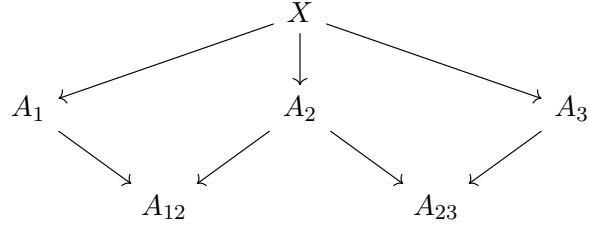
\[\begin{tikzcd}[ampersand replacement=\&]
	\&\& X \&\& \\
	{A_1} \&\& {A_2} \&\& {A_3} \\
	\& {A_{12}} \&\& {A_{23}}
	\arrow[from=1-3, to=2-1]
	\arrow[from=1-3, to=2-3]
	\arrow[from=1-3, to=2-5]
	\arrow[from=2-1, to=3-2]
	\arrow[from=2-3, to=3-2]
	\arrow[from=2-3, to=3-4]
	\arrow[from=2-5, to=3-4]
\end{tikzcd}\]

then this is a multipullback diagram of $C^*$-algebras i.e. the induced map $\sigma : X \xrightarrow[]{} \bigoplus_{ij} A_i$ given by $\sigma(x) = (\pi_i(x))_i$ is an isomorphism iff

the following hold:
\begin{enumerate}
    \item $\bigcap_{j \in I} ker\pi_j = \{0\}$
    \item $\pi_2(X) = \bigcap_{j \neq 2} (\pi_j^2)^{-1}(\pi_2^j(A_j))$
    \item $ker \pi_2^k = \pi_k(\bigcap_{p \neq k} ker \pi_p)$ for all $k \neq 2$
\end{enumerate}
\label{multi}
\end{proposition}
\begin{proof}
   
This proof is heavily based on the original proof by Pedersen in \cite[Proposition 3.1]{pedersen_pullback_1999} for a pullback of $C^*$-algebras, and has been adapted here to the case of a multipullback of $C^*$-algebras.

First let us assume the three conditions hold.

Then $\sigma : X \xrightarrow[]{} \bigoplus_{ij} A_i$ maps $x \mapsto (\pi_i(x))_i$. 

Say we have a $x \in X$ such that $\sigma(x) = 0$, then $\pi_i(x) = 0$ for all $i$, then $x$ is in $\bigcap_{j \in I} ker\pi_j$ which is the singleton $\{0\}$, so we have that $\sigma$ is injective.

For surjectivity,

say $(a_i)_{i\in I}$, then,

\begin{align*}
    a_2 &\in (\pi_j^2)^{-1} (\pi_2^j(a_j)) \\
    a_2 &\in (\pi_j^2)^{-1} (\pi_2^j(A_j))
\end{align*}
So by (2) we have that $a_2 = \pi_2(x)$ for some $x \in X$.

Now consider $k \neq 2$.

\begin{align*}
    \pi_2^k(a_k - \pi_k(x)) &= \pi_2^k(a_k) - \pi_2^k(\pi_k(x)) \\
                            &= \pi_k^2(a_2) - \pi_k^2(\pi_2(x)) \\
                            &= 0
\end{align*}

We have that
\begin{align*}
    a_k - \pi_k(x) &\in ker\pi_2^k \\
    a_k - \pi_k(x) &= \pi_k(y_k) \hspace{2cm} \text{ for some } y_k \in \bigcap_{p \neq k} ker \pi_p
\end{align*}

Clearly, we have that $\sigma( x + \Sigma_k y_k) = (a_i)_i$ and we conclude that $\sigma$ is surjective.

For the implication in the other direction:

Let $x \in \bigcap_{j \in I} ker \pi_j$, then $\sigma(x) = 0$, but as $\sigma$ is injective, we have that $x = 0$ and we get that (1) is satisfied.
Let $\sigma$ be surjective, then to see that (2) holds, we observe that,

if $a_2 = \pi_2(x)$ for some $x \in X$, then $\pi_j^2(a_2) = \pi_2^j(\pi_j(x)) \implies a_2 \in (\pi_j^2)^{-1}(\pi_2^j(A_j))$ and if $a_2 \in (\pi_j^2)^{-1}(\pi_2^j(A_j))$, then $(a_i)_i \in \oplus_{ij} A_i$, so by surjectivity of $\sigma$ we have that $a_2 = \pi_2(x)$ for some $x \in X$.

Now if
\begin{align*}
     &a_k \in ker \pi_2^k \\
    \implies &\pi_2^k(a_k) = 0 \\
    \implies &(a_k,0,0) \in \oplus_{ij}A_i \\
    \implies &\exists x \in X \ni \pi_k(x) = a_k \text{ and } \pi_p(x) = 0 \ \forall \ p \neq x
\end{align*}

The other inequality for (3) follows easily.
\end{proof}

\begin{remark}
The conditions given in the proposition are sufficient for a general multipullback diagram, i.e., a diagram with all $A_{ij}$'s. However, the conditions are not necessary in this case and we need the star shaped diagram for the necessity part. Furthermore, imposing the cocycle condition \cite[Defintion 6]{hajac_cocycle_2012} resolves the necessity of condition (2), however, it doesn't imply the necessity of condition (3). 
\end{remark}

\begin{theorem}
The $C^*$-algebra $C(S_H^{2n})$ is isomorphic to the wide pullback $\sigma_{-1} : \mathcal{T}^n \xrightarrow[]{} C(S^{2n-1}_H)$, $\sigma_{1} : \mathcal{T}^n \xrightarrow[]{} C(S^{2n-1}_H)$ and $ev_{-1} : C(S_H^{2n-1})\otimes C[-1,1] \xrightarrow[]{} C(S^{2n-1}_H)$ and $ev_{1} : C(S_H^{2n-1})\otimes C[-1,1] \xrightarrow[]{} C(S^{2n-1}_H)$ in the following manner:
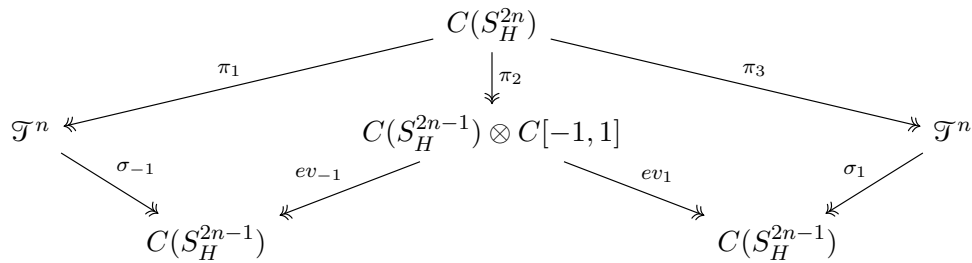
\[\begin{tikzcd}[ampersand replacement=\&]
	\&\& {C(S_H^{2n})} \&\& \\
	{\mathcal{T}^n} \&\& {C(S_H^{2n-1}) \otimes C[-1,1]} \&\& {\mathcal{T}^n} \\
	\& {C(S^{2n-1}_H)} \&\& {C(S^{2n-1}_H)}
	\arrow["{\pi_1}"', two heads, from=1-3, to=2-1]
	\arrow["{\pi_2}", two heads, from=1-3, to=2-3]
	\arrow["{\pi_3}", two heads, from=1-3, to=2-5]
	\arrow["{\sigma_{-1}}", two heads, from=2-1, to=3-2]
	\arrow["{ev_{-1}}"', two heads, from=2-3, to=3-2]
	\arrow["{ev_1}", two heads, from=2-3, to=3-4]
	\arrow["{\sigma_1}"', two heads, from=2-5, to=3-4]
\end{tikzcd}\]
where the maps $\sigma_1$ and $\sigma_{-1}$ are the $\sigma_n$ as in Proposition \ref{bdos} and the maps $\pi_i$ are defined as:

\begin{align*}
    \pi_1(s_i) &= t_i \hspace{3cm} 0 \leq i < n\\
    \pi_1(s_n) &= -1\\
    \pi_2(s_i) &= s_i \otimes 1 \hspace{3cm} 0 \leq i < n\\
    \pi_2(s_n) &= 1 \otimes z\\
    \pi_3(s_i) &= t_i \hspace{3cm} 0 \leq i < n\\
    \pi_3(s_n) &= 1\\
\end{align*}
\end{theorem}
\begin{proof}
We need to check that conditions (1)-(3) of Proposition \ref{multi} are satisfied.

(1) 

$ker\pi_1$ is the ideal generated by $(1+s_n)$, $ker\pi_3$ is the ideal generated by $(1-s_n)$ while the kernel of $\pi_2$ is generated by $\prod_{0 \leq i <n}(1-s_is_i^*)$. So as $s_n$ is central, we have that $\bigcap ker\pi_i = \langle\prod (1-s_is_i^*)(1-s_n^2) \rangle$ which is $0$ by the last generating relation of $C(S^{2n}_H)$.

(2)

All the maps are surjective, so this condition follows trivially.

(3)
$ker\pi_2 \cap ker\pi_3$ is inside the ideal generated by $(1-s_n)\prod(1-s_is_i^*)$ which maps to $2\prod(1-s_is_i^*) = ker\sigma_{-1}$. Furthermore, if $a \in \langle \prod (1-t_it_i^*) \rangle$ and due to surjectivity of $\pi_1$, we have that there exists a $b \in \langle \prod (1-s_is_i^*) \rangle$ such that $\pi_1(b) = a$, then consider the element $\pi_1(b(1-s_n)/2) = a$. $\pi_2 (\prod(1 - s_is_i^*)) = 0$ by the generating relation for $C(S_H^{2n-1})$, and $\pi_3(1-s_n) = 0$, so we have the other inequality. 
\end{proof}

\begin{remark}
This shows that $C(S_H^{2n})$ is a $K$-weak suspension of $C(S_H^{2n-1})$, as $\mathcal{T}^n$ has the same $K$-theory as that of a point, and replacing $\mathcal{T}^n$ with the $C^*$-algebra of a  point i.e. $\mathbb{C}$ yields the usual suspension of $C^*$-algebras. 
\end{remark}

\section{Acknowledgements}
This research is part of the EU Staff Exchange project 101086394 Operator Algebras That One Can See. The project is co-financed by the Polish Ministry of Education and Science under the program PMW (grant agreement 5305/HE/2023/2). The author is grateful to Piotr M. Hajac for numerous conversations regarding multipullback $C^*$-algebras. The author is also grateful to Francesco D'Andrea and Elmar Wagner for discussions on the quantum real projective spaces. Finally, the author would also like to extend gratitude to Elmar Wagner and Universidad Michoacana de San Nicolás de Hidalgo, Masoud Khalkali and Western University, George Elliot and Fields Insitute for their kind hospitality. The author would further want acknowledge the use of Claude 5.5 Opus for grammar checking and copy editing. No proof was generated using generative AI and the author would like to claim full responsibility of the results presented.
\bigskip
\printbibliography

@article{hajac_quantum_2003,
	title = {Quantum Real Projective Space, Disc and Spheres},
	volume = {6},
	issn = {1572-9079},
	url = {https://doi.org/10.1023/A:1023288309786},
	doi = {10.1023/A:1023288309786},
	pages = {169--192},
	number = {2},
	journaltitle = {Algebras and Representation Theory},
	shortjournal = {Algebras and Representation Theory},
	author = {Hajac, Piotr M. and Matthes, Rainer and Szymański, Wojciech},
	urldate = {2025-08-12},
	date = {2003-05-01},
	langid = {english},
}

@article{katsura_class_2004,
	title = {A class of 𝐶*-algebras generalizing both graph algebras and homeomorphism 𝐶*-algebras I, fundamental results},
	volume = {356},
	issn = {0002-9947, 1088-6850},
	url = {https://www.ams.org/tran/2004-356-11/S0002-9947-04-03636-0/},
	doi = {10.1090/S0002-9947-04-03636-0},
	pages = {4287--4322},
	number = {11},
	journaltitle = {Transactions of the American Mathematical Society},
	shortjournal = {Trans. Amer. Math. Soc.},
	author = {Katsura, Takeshi},
	urldate = {2024-05-27},
	date = {2004},
	langid = {english},
}

@article{hajac_k-theory_2018,
	title = {The \$K\$-theory of twisted multipullback quantum odd spheres and complex projective spaces},
	volume = {12},
	issn = {1661-6952, 1661-6960},
	url = {https://ems.press/doi/10.4171/jncg/292},
	doi = {10.4171/jncg/292},
	pages = {823--863},
	number = {3},
	journaltitle = {Journal of Noncommutative Geometry},
	shortjournal = {J. Noncommut. Geom.},
	author = {Hajac, Piotr M. and Nest, Ryszard and Pask, David and Sims, Aidan and Zieliński, Bartosz},
	urldate = {2024-01-21},
	date = {2018-07-27},
}

@article{hong_noncommutative_2008,
	title = {Noncommutative balls and mirror quantum spheres},
	volume = {77},
	rights = {© 2008 London Mathematical Society},
	issn = {1469-7750},
	url = {https://onlinelibrary.wiley.com/doi/abs/10.1112/jlms/jdn003},
	doi = {10.1112/jlms/jdn003},
	pages = {607--626},
	number = {3},
	journaltitle = {Journal of the London Mathematical Society},
	author = {Hong, Jeong Hee and Szymański, Wojciech},
	urldate = {2023-11-03},
	date = {2008},
	langid = {english},
	note = {\_eprint: https://onlinelibrary.wiley.com/doi/pdf/10.1112/jlms/jdn003},
}

@article{pedersen_pullback_1999,
	title = {Pullback and Pushout Constructions in C*-Algebra Theory},
	volume = {167},
	issn = {0022-1236},
	url = {https://www.sciencedirect.com/science/article/pii/S0022123699934560},
	doi = {10.1006/jfan.1999.3456},
	pages = {243--344},
	number = {2},
	journaltitle = {Journal of Functional Analysis},
	shortjournal = {Journal of Functional Analysis},
	author = {Pedersen, Gert K.},
	urldate = {2023-05-04},
	date = {1999-10-01},
	langid = {english},
}

@article{hong_quantum_2002,
	title = {Quantum Spheres and Projective Spaces as Graph Algebras},
	volume = {232},
	issn = {1432-0916},
	url = {https://doi.org/10.1007/s00220-002-0732-1},
	doi = {10.1007/s00220-002-0732-1},
	pages = {157--188},
	number = {1},
	journaltitle = {Communications in Mathematical Physics},
	shortjournal = {Commun. Math. Phys.},
	author = {Hong, Jeong Hee and Szymański, Wojciech},
	urldate = {2022-04-22},
	date = {2002-12-01},
	langid = {english},
}

@misc{saurabh_cso_q2n1so_q2n-1_2026,
	title = {\$C({SO}\_q(2n+1)/{SO}\_q(2n-1))\$ as iterated torsioned quantum double suspensions of \$C({\textbackslash}mathbb\{T\})\$},
	url = {http://arxiv.org/abs/2602.17075},
	doi = {10.48550/arXiv.2602.17075},
	number = {{arXiv}:2602.17075},
	publisher = {{arXiv}},
	author = {Saurabh, Bipul},
	urldate = {2026-08-10},
	date = {2026-02-19},
	eprinttype = {arxiv},
	eprint = {2602.17075 [math.OA]},
}

@misc{dandrea_quantum_2026,
	title = {Quantum {CW}-complexes in a Waldhausen category for unital C*-algebras},
	rights = {{arXiv}.org perpetual, non-exclusive license},
	url = {https://arxiv.org/abs/2002.09015},
	doi = {10.48550/ARXIV.2002.09015},
	publisher = {{arXiv}},
	author = {D'Andrea, Francesco and Hajac, Piotr M. and Maszczyk, Tomasz and Zielinski, Bartosz},
	urldate = {2026-09-19},
	date = {2026-09-10},
	note = {Version Number: 4},
}

@article{chakraborty_quantum_2011,
	title = {Quantum double suspension and spectral triples},
	volume = {260},
	rights = {https://www.elsevier.com/tdm/userlicense/1.0/},
	issn = {00221236},
	url = {https://linkinghub.elsevier.com/retrieve/pii/S0022123611000255},
	doi = {10.1016/j.jfa.2011.01.009},
	pages = {2716--2741},
	number = {9},
	journaltitle = {Journal of Functional Analysis},
	shortjournal = {Journal of Functional Analysis},
	author = {Chakraborty, Partha Sarathi and Sundar, S.},
	urldate = {2026-09-25},
	date = {2011-05},
	langid = {english},
}

@article{katsura_class_2006,
	title = {A {CLASS} {OF} C*-{ALGEBRAS} {GENERALIZING} {BOTH} {GRAPH} {ALGEBRAS} {AND} {HOMEOMORPHISM} C*-{ALGEBRAS} {II}, {EXAMPLES}},
	volume = {17},
	issn = {0129-167X, 1793-6519},
	url = {https://www.worldscientific.com/doi/abs/10.1142/S0129167X06003722},
	doi = {10.1142/S0129167X06003722},
	pages = {791--833},
	number = {7},
	journaltitle = {International Journal of Mathematics},
	shortjournal = {Int. J. Math.},
	author = {Katsura, Takeshi},
	urldate = {2026-09-25},
	date = {2006-08},
	langid = {english},
}

@article{katsura_class_2006-1,
	title = {A class of \$\{C{\textasciicircum}*\}\$-algebras generalizing both graph algebras and homeomorphism \$\{C{\textasciicircum}*\}\$-algebras {III}, ideal structures},
	volume = {26},
	rights = {https://www.cambridge.org/core/terms},
	issn = {0143-3857, 1469-4417},
	url = {https://www.cambridge.org/core/product/identifier/S0143385706000320/type/journal_article},
	doi = {10.1017/S0143385706000320},
	pages = {1805--1854},
	number = {6},
	journaltitle = {Ergodic Theory and Dynamical Systems},
	shortjournal = {Ergod. Th. Dynam. Sys.},
	author = {Katsura, Takeshi},
	urldate = {2026-09-25},
	date = {2006-12},
	langid = {english},
}

@article{gothe_gluing_2026,
	title = {Gluing topological graph C∗-algebras},
	issn = {1793-5253, 1793-7167},
	url = {https://www.worldscientific.com/doi/10.1142/S1793525326500196},
	doi = {10.1142/S1793525326500196},
	pages = {1--29},
	journaltitle = {Journal of Topology and Analysis},
	shortjournal = {J. Topol. Anal.},
	author = {Gothe, Atul and Quigg, John and Tobolski, Mariusz},
	urldate = {2026-09-25},
	date = {2026-02-07},
	langid = {english},
}

@article{katsura_c-algebras_2004,
	title = {On C∗-algebras associated with C∗-correspondences},
	volume = {217},
	rights = {https://www.elsevier.com/tdm/userlicense/1.0/},
	issn = {00221236},
	url = {https://linkinghub.elsevier.com/retrieve/pii/S002212360400093X},
	doi = {10.1016/j.jfa.2004.03.010},
	pages = {366--401},
	number = {2},
	journaltitle = {Journal of Functional Analysis},
	shortjournal = {Journal of Functional Analysis},
	author = {Katsura, Takeshi},
	urldate = {2026-09-25},
	date = {2004-12},
	langid = {english},
}

@article{hajac_strong_1996,
	title = {Strong connections on quantum principal bundles},
	volume = {182},
	rights = {https://www.springer.com/tdm},
	issn = {0010-3616, 1432-0916},
	url = {https://link.springer.com/10.1007/BF02506418},
	doi = {10.1007/BF02506418},
	pages = {579--617},
	number = {3},
	journaltitle = {Communications in Mathematical Physics},
	shortjournal = {Commun.Math. Phys.},
	author = {Hajac, Piotr M.},
	urldate = {2026-09-25},
	date = {1996-12},
	langid = {english},
}

@article{hajac_covariant_2024,
	title = {The covariant functoriality of graph algebras},
	volume = {56},
	issn = {0024-6093, 1469-2120},
	url = {https://londmathsoc.onlinelibrary.wiley.com/doi/10.1112/blms.13125},
	doi = {10.1112/blms.13125},
	pages = {3269--3288},
	number = {11},
	journaltitle = {Bulletin of the London Mathematical Society},
	shortjournal = {Bulletin of London Math Soc},
	author = {Hajac, Piotr M. and Tobolski, Mariusz},
	urldate = {2026-09-25},
	date = {2024-11},
	langid = {english},
}

@article{schochet_topological_1982,
	title = {Topological methods for \textit{C}$^{\textrm{∗}}$ -algebras. {II}. Geometry resolutions and the Künneth formula},
	volume = {98},
	issn = {0030-8730, 0030-8730},
	url = {http://msp.org/pjm/1982/98-2/p15.xhtml},
	doi = {10.2140/pjm.1982.98.443},
	pages = {443--458},
	number = {2},
	journaltitle = {Pacific Journal of Mathematics},
	shortjournal = {Pacific J. Math.},
	author = {Schochet, Claude},
	urldate = {2026-09-25},
	date = {1982-02-01},
	langid = {english},
}

@article{paulovicks_cuntz-krieger_2025,
	title = {A Cuntz-Krieger uniqueness theorem for Cuntz-Pimsner algebras},
	volume = {131},
	issn = {1903-1807, 0025-5521},
	url = {https://journals.msp.org/mscand/article/view/21},
	doi = {10.7146/math.scand.a-156724},
	number = {2},
	journaltitle = {{MATHEMATICA} {SCANDINAVICA}},
	shortjournal = {Math. Scand.},
	author = {Paulovicks, Menevçse Eryüzlü and Tomforde, Mark},
	urldate = {2026-09-25},
	date = {2025-07-22},
}

@article{bates_c-algebras_2000,
	title = {The {\textbackslash}[C{\textasciicircum}\{*\}{\textbackslash}]-algebras of row-finite graphs},
	volume = {6},
	pages = {307--324},
	journaltitle = {New York Journal of Mathematics},
	publisher = {New York Journal of Mathematics},
	author = {Bates, Teresa and Pask, David and Raeburn, Iain and Szymański, Wojciech},
	date = {2000},
}

@collection{voiculescu_free_1996,
	location = {Providence, Rhode Island},
	title = {Free Probability Theory},
	isbn = {978-1-4704-2980-5},
	url = {https://www.ams.org/fic/012},
	doi = {10.1090/fic/012},
	publisher = {American Mathematical Society},
	editor = {Voiculescu, Dan-Virgil},
	urldate = {2026-09-26},
	date = {1996-12-10},
	langid = {english},
}

@thesis{foster_quantum_2005,
	title = {Quantum odd-dimensional balls},
	institution = {University of Newcastle},
	type = {Honours thesis},
	author = {Foster, J.D.},
	date = {2005},
}

@article{pedersen_stability_1998,
	title = {Stability of anticommutation relations: An Application of noncommutative {CW}-complexes},
	volume = {1998},
	issn = {0075-4102, 1435-5345},
	url = {https://www.degruyter.com/document/doi/10.1515/crll.1998.055/html},
	doi = {10.1515/crll.1998.055},
	shorttitle = {Stability of anticommutation relations},
	pages = {101--143},
	number = {499},
	journaltitle = {Journal für die reine und angewandte Mathematik (Crelles Journal)},
	author = {Pedersen, Gert K. and Loring, Terry A. and Eilers, Søren},
	urldate = {2026-10-01},
	date = {1998-06-01},
}

@article{hajac_cocycle_2012,
	title = {Cocycle condition for multi-pullbacks of algebras},
	volume = {98},
	issn = {0137-6934, 1730-6299},
	url = {http://journals.impan.pl/cgi-bin/doi?bc98-0-9},
	doi = {10.4064/bc98-0-9},
	pages = {239--243},
	journaltitle = {Banach Center Publications},
	shortjournal = {Banach Center Publ.},
	author = {Hajac, Piotr M. and Zieliński, Bartosz},
	urldate = {2026-10-15},
	date = {2012},
	langid = {english},
}

\bigskip
\vspace{10cm}
\Addresses

\end{document}